\documentclass[a4paper,twoside]{article}

\usepackage{amssymb,amsmath,amsthm}
\usepackage[toc,page]{appendix}
\usepackage{texdraw}

\def \RR {{\mathbb R}}
\def \NN {{\mathbb N}}
\def \ZZ {{\mathbb Z}}

\def \QQ {{\mathbb Q}}
\def \SS {{\mathbb S}}

\def \cont{{\mathcal C}}

\newtheorem{theorem}{Theorem}[section]
\newtheorem{lemma}[theorem]{Lemma}
\newtheorem{example}[theorem]{Example}

\newtheorem{proposition}[theorem]{Proposition}
\newtheorem{remark}[theorem]{Remark}

\def \ds{\displaystyle}

\title{\sc On the One-Dimentional Pompeiu Problem}
\author{%
Vivina Barutello\footnote{Dipartimento di Matematica, Universit\`a degli Studi
di Torino, Via Carlo Alberto, 10,  10123 Torino,
Italy. e-mail: \texttt{vivina.barutello@unito.it, camillo.costantini@unito.it}}
\and
Camillo Costantini\footnotemark[1]
}

\date{}
\begin{document}
%%%%%%%%%%%%%%%%%%%%%%%%%%%%%%%%%%%%%%%%%%%%%%%%%%%%%%%%%%%%%%%%%%%%%%%%%%%%%%%
\maketitle
%%%%%%%%%%%%%%%%%%%%%%%%%%%%%%%%%%%%%%%%%%%%%%%%%%%%%%%%%%%%%%%%%%%%%%%%%%%%%%%
\begin{abstract}
We investigate the Pompeiu property for subsets of the real line, under no assumption of connectedness.
In particular we focus our study on finite unions of bounded (disjoint) intervals, and we emphasize 
the different results corresponding to the cases where the function in question is supposed to have constant integral on all isometric images, or just on all the translation-images of the domain. While no set of the previous kind enjoys the Pompeiu property in the latter sense, we provide a necessary and sufficient condition in order a union of two intervals to have the Pompeiu property in the former sense, and we produce some examples to give an insight of the complexity of the problem for three-interval sets.
\end{abstract}

%======================
\section{Introduction}
%=====================
The Pompeiu problem traces back to 1929, and has been one of the most
extensively investigated issues both in applied and abstract mathematics.
Even if the original formulation given by Pompeiu in his basic papers \cite{Pompeiu1,Pompeiu3,Pompeiu2}
was including some supplementary assumptions, 
nowadays the vague appellation of \emph{Pompeiu problem} 
may label any question which sounds like this:
\begin{center}
\emph{
\parbox{5in}{
Let  $D \subseteq \RR^n$ be a measurable set and $f$ 
a continuous real-valued function on $\RR^n$ whose integral 
on every set ``congruent'' to $D$ takes a constant value $c$.
Must then the function $f$ be itself constant?
}}
\end{center}
If a domain $D\subseteq\RR^n$ is such that the above question is answered in the positive for every continuous function $f\colon\RR^n\rightarrow\RR$ satisfying the assumption, then
$D$ is said to have the \emph{Pompeiu property} (of course, this depends also on the definition of ``congruent'' we are considering). In the literature, at our best knowledge, all papers devoted to the Pompeiu problem are concerned with the case where $D$ is convex, or at least connected. For example, it is well-known that, when considering 
\emph{rigid motions} (translations composed with rotations), any ball in $\RR^n$ fails to enjoy the Pompeiu property, while it holds for some classes of domains whose boundary is $\SS^{n-1}$ (as ellipses and regular polygons in the plane, see \cite{Brown1982,Brown1968,Brown1973} and the survey paper \cite{Zalcman1992}).\par
As far as only connected domains are investigated, the Pompeiu problem is of no interest in $\RR$: clearly, for every bounded interval $[a,b]$, the function $f(x)=\sin\big(\frac{2\pi}{b-a}x\big)$ is non-constant and such that its integral is $0$ on every subset of $\RR$ congruent to $[a,b]$ (in fact, this is just a special case of the 
above-mentioned result, that no ball in $\RR^n$ may enjoy the Pompeiu property). However, once the connectedness assumption is dropped, the one-dimensional case becomes non-trivial, and this corresponds exactly to the kind of  investigation carried out in the present paper.\par
More specifically, our study has been focused on the following situation. Let $I$ be a finite union of disjoint bounded intervals of the real line, let $\Sigma$ be the set of all isometries $\sigma\colon I\rightarrow \RR$, and $\Sigma'$ a distinguished subset of $\Sigma$ (with, possibly, $\Sigma'=\Sigma$): when is it true that for every continuous function $f\colon\RR\rightarrow\RR$, the implication
\begin{equation} \label{eq:PC}
\left( \exists C \in \RR   \quad \forall \sigma \in \Sigma'   \quad \int_{\sigma(I)}f = C  \right)
\quad \implies \quad
f \text{ \emph{constant}}
\end{equation}
holds?\par
It is to be observed that, by a general result concerning the geometric structure of the spaces $\RR^n$, if $A$ is any subset of $\RR^n$ and $j\colon A\rightarrow \RR^n$ is an isometry, then $j$ extends to an isometry $j'\colon\RR^n\rightarrow\RR^n$ (equivalently, every isometry from a subset of $\RR^n$ to $\RR^n$ is the restriction of some isometry from the whole of $\RR^n$ (on)to $\RR^n$). Therefore, in the previous formulation we may equally well let $\Sigma$ to be the set of all isometries of the real line onto itself (i.e., as is well-known, of all translations possibly composed with a reflection). In our study, we have tackled question~\eqref{eq:PC} for two subsets 
$\Sigma'$ of $\Sigma$: namely, $\Sigma'$ $=$ set of all translations of the real line, and $\Sigma'=\Sigma$ (the set itself of the isometries). The results obtained in the two cases have turned out to be quite different. \par
In Proposition \ref{propo:false} \emph{infra} we prove that, whenever $D$ is the union of two bounded disjoint intervals,
it fails to enjoy the Pompeiu property with respect to $\Sigma'$ $=$ set of all translations of the real line. In fact, the argument used for the proof outlines an inductive procedure to obtain a non-constant real function whose integral on every translation-image of $D$ is constant.
As we point out in Remark \ref{rem:unioni_finite1}, the result extends to any finite union of disjoint bounded intervals, by an analogous but technically heavier (and tedious) proof. \par
On the other hand, when taking $\Sigma' = \Sigma$, even for the union of two disjoint intervals the situation appears to be multi-faceted, and the answer to the basic question depends on the relationships between the three fundamental quantities involved: the length of the two intervals and their gap. Theorem \ref{teo:nec_suf} \emph{infra} gives a necessary and sufficient condition for $D$ to enjoy the Pompeiu property in this case; in particular, the statement emphasizes the crucial r\^ole played by the rational or irrational character of some ratios related to the three quantities above.\par
Contrary to the results obtained for translations, when $\Sigma'=\Sigma$ passing from two to three intervals considerably boosts the complexity of the problem. In this paper we do not investigate in detail the three-interval (nor the more-interval) case. However, we prove that from every two-interval set (even enjoying the Pompeiu property) we may always obtain, by adding a third interval, a set for which the property does not hold; and we also give an example of a situation where the opposite phenomenon happens. The former result appears, in particular, to be somehow anti-intuitive, as it could seem reasonable that increasing the complexity of the set, the probability of getting the Pompeiu property increases as well. This should show how interesting and probably twisted would be a systematic study of the Pompeiu problem for multi-interval sets, when taking $\Sigma'=\Sigma$.

%===============================================
\section{The two-interval case}
\label{sec:1}
%==============================================
In this section we study conjecture \eqref{eq:PC} stated in the introduction, 
when we deal with continuous functions and
$I$ is the disjoint union of two non-trivial compact intervals. \\
In a first result we take into account a non-trivial subset of $\Sigma$, indeed we consider 
only translations on the real line: in this case \eqref{eq:PC} is false, and we will prove
the existence on infinitely many non-constant functions satisfying the integral condition. 
On the other hand, when also reflections are allowed, that is $\Sigma' \equiv \Sigma$, we 
will find a necessary and sufficient condition on $I$ in order to obtain a positive answer. 

We start proving that the integral condition, when $\sigma$ varies in the set of translations, 
is equivalent to a pointwise one.

\begin{lemma} \label{lem:cond_puntuale}
Let $a<b<c<d$ and $f \in \cont(\RR)$.
Then the following two conditions are equivalent:
\begin{itemize}
\item[(I)] $\displaystyle F(t) := \int_{a+t}^{b+t}f(x)dx + \int_{c+t}^{d+t}f(x)dx$
           is constant as $t$ varies in $\RR$;
\item[(P)] $f(a+t) + f(c+t) = f(b+t) + f(d+t)$, for every $t \in \RR$, i.e.
\begin{equation}\label{eq:cond_punt}
\forall x \in \RR, f(x) = f(x+a-d) + f(x+c-d) - f(x+b-d).
\end{equation}
\end{itemize}
\end{lemma}

\begin{proof}
\noindent (I) $ \Rightarrow$ (P) Trivially follows deriving the constant function $F$.

\noindent (P) $ \Rightarrow$ (I) Let $t',t'' \in \RR$ with $t'<t''$, and set, for the sake of semplicity
$s := t''-t'$, $a':= a+t'$, $b':= b+t'$, $c':= c+t'$, $d':= d+t'$
and
\[
r' := \int_{a+t'}^{b+t'}f(x)dx + \int_{c+t'}^{d+t'}f(x)dx 
   = \int_{a'}^{b'}f(x)dx + \int_{c'}^{d'}f(x)dx.
\]
Our aim is to prove that
\[
r'' := \int_{a+t''}^{b+t''}f(x)dx + \int_{c+t''}^{d+t''}f(x)dx =r'. 
\]
Using assumption \eqref{eq:cond_punt} and the definition of $r'$, we see that
\begin{equation}\label{eq:..}
\begin{split}
r'' & = \int_{a'+s}^{b'+s}f(x)dx + \int_{c'+s}^{d'+s}f(x)dx \\
    & = r' - \int_{a'}^{a'+s}f(x)dx + \int_{b'}^{b'+s}f(x)dx 
           - \int_{c'}^{c'+s}f(x)dx \\
    & \quad \quad+ \left( \int_{d'}^{d'+s} [f(x+a-d) + f(x+c-d) - f(x+b-d)] dx \right). 
\end{split}
\end{equation}
Now, straightforward changes of variables show that
\[
\begin{split}
\int_{d'}^{d'+s} f(x+a-d) dx &= \int_{d'+a-d}^{d'+s+a-d} f(z) dz = \int_{a'}^{a'+s} f(z) dz, \\ 
\int_{d'}^{d'+s} f(x+c-d) dx &= \int_{c'}^{c'+s} f(z) dz,\\ 
\int_{d'}^{d'+s} f(x+b-d) dx &= \int_{b'}^{b'+s} f(z) dz,
\end{split}
\]
and we conclude by replacing in \eqref{eq:..}.
\end{proof}

\begin{remark}
The previous lemma still holds when $f \in L^1_{loc}(\RR)$ if
we read the pointwise equality on $\RR$ exept a zero-measure set.
\end{remark}

\begin{proposition}\label{propo:false}
Let $I=[a,b]\cup[c,d]$, for some $a<b<c<d$, 
and $\Sigma'$ $=$ set of translations of the real line, then conjecture 
\eqref{eq:PC} is false in the realm of continuous functions.
\end{proposition}

\begin{proof}
Let $f_0 \in \cont([a,d])$ be such that
\[
f_0(d) = f_0(a)+f_0(c)-f_0(b) \quad \text{and} \quad \int_I f_0(x)dx =C.
\]
We now consider the sequence of functions $(f_n)_{n\geq 1}$ 
defined by the recurrence relation:
\[
f_n(x) :\begin{cases}
& f_{n-1}(x+b-a) + f_{n-1}(x+c-a) - f_{n-1}(x+d-a), \\
& \qquad \qquad \quad \text{if } x \in [a-n(b-a),a-(n-1)(b-a)), \\
& f_{n-1}(x), \quad \text{if } x \in [a-(n-1)(b-a),d+(n-1)(d-c)], \\
& f_{n-1}(x+a-d) + f_{n-1}(x+c-d) - f_{n-1}(x+b-d), \\
& \qquad \qquad \quad \text{if } x \in (d+(n-1)(d-c),d+n(d-c)].
\end{cases}
\]
Actually, for every $n$, $f_n$ is an extension of $f_{n-1}$; 
it is then streightforward to verify that each $f_n$ is well defined and continuous on
$[a-n(b-a),d+n(d-c)]$, and that $f_n$ converges to some $f \in \cont(\RR)$. Such limit function,
by defintition, also verify the pointwise relation \eqref{eq:cond_punt}, hence, 
by Lemma \ref{lem:cond_puntuale}, its integral on $\sigma(I)$ does not depend on 
the translation $\sigma$. We conclude observing that, when $\sigma$ is the identity
\[
\int_I f(x)dx = \int_I f_0(x)dx = C.
\]
\end{proof}

\begin{remark}\label{rem:unioni_finite1}
Both Lemma \ref{lem:cond_puntuale} and Proposition \ref{propo:false}
still hold when $I$ is the disjoint union of more then two compact intervals,
that is 
\[
I = \bigcup_{i=1}^N [a_i,b_i], \quad a_1 < b_1 < a_2 < \ldots < a_N < b_N.
\]
We omit the proofs since they are the perfect analogue of the ones we have proposed.
\end{remark}

When the set $\Sigma'$ coincide with the set $\Sigma$ of all isometries of the real line
(i.e., translations, reflections and their compositions), then the situation turns out to be 
quite different and, instead of the negative result of Proposition \ref{propo:false}, 
we obtain a necessary 
and sufficient condition  in order the Pompeiu's conjecture to hold.
In a similar way to the statement of  Proposition \ref{propo:false},
in the argument we are going to carry out we will consider a set $I=[a,b]\cup[c,d]$,
with $a<b<c<d$ arbitrarily fixed real numbers.
Since the problem is clearly isometric-invariant,
the fundamental quantities involved are the lengths of the two intervals $[a,b]$,$[c,d]$ 
and of the hole between them.
Thus, we define
\[
\ell := b-a, \qquad H := c-b, \qquad L := d-c,
\] 
which yields the equality $I=[0,\ell] \cup [\ell+H,\ell+H+L]$.
Moreover, since reflections are also allowed, it is not restrictive 
to assume that $L\geq\ell$.
Furthermore, given $\alpha < \beta < \gamma < \delta$ and a (measurable) function $f:\RR \to \RR$,
we define
\[ \label{notazione}
[\lambda^{\pm}, \xi, \Lambda^{\pm};\alpha] = [\lambda^{\pm}, \xi, \Lambda^{\pm};\alpha]_f := 
\pm\int_{\alpha}^{\beta} f(x) dx \pm\int_{\gamma}^{\delta} f(x) dx,
\]
whenever $\lambda := \beta-\alpha$, $\Lambda := \delta-\gamma$ and $\xi := \gamma -\beta$.
With this notation 
\begin{equation}\label{eq:hpconj}
\left( \exists C \in \RR   \quad \forall \sigma \in \Sigma   \quad \int_{\sigma(I)}f = C  \right)
\quad \iff \quad
\left( 
\begin{array}{c}
\exists C \in \RR   \quad \forall x \in \RR   \\ 
\\ 
\left[\ell^+,H,L^+; x\right] = \left[L^+,H,\ell^+;x\right]=C
\end{array}
\right).
\end{equation}
\begin{lemma}\label{lem:hole}
If $f:\RR \to\RR$ is such that one of the two equivalent assumptions of \eqref{eq:hpconj}
holds (with $\ell,L,H$ as above), then letting $H'=3H+L+\ell$ we have the equalities
\[
\left[\ell^+,H',L^+;x\right] = C = \left[L^+,H',\ell^+;x\right] \quad \text{for every }x \in \RR.
\]
Moreover,
\begin{equation}\label{eq:equiv}
\frac{L-\ell}{L+H} \in \QQ 
\quad \iff \quad 
\frac{L-\ell}{L+H'} \in \QQ.
\end{equation}
\end{lemma}
\begin{proof}
On the one hand we have, for every $x \in \RR$,
\[
[\ell^+, H,  L^+;x] =C \quad \text{and} \quad [L^+, H, l^+;\ell+H+x] =C;
\]
hence, subtracting term by term the latter equation from the former one (we refer to Figure \ref{fig:STEP1}),
\[
[\ell^+, H+L+H, \ell^-;x]=0 \quad \text{for every } x \in \RR.
\]
On the other hand, since $[\ell^+, H, L^+;\ell+H+L+H+x] = C$ also holds for any $x \in \RR$, 
then summing term by term we obtain the first of the required equalities
\[
[\ell^+,H+L+H+\ell+H,L^+;x] = C.
\]
In a completely symmetric way, it is also proved the second one.
\begin{figure}[t!]
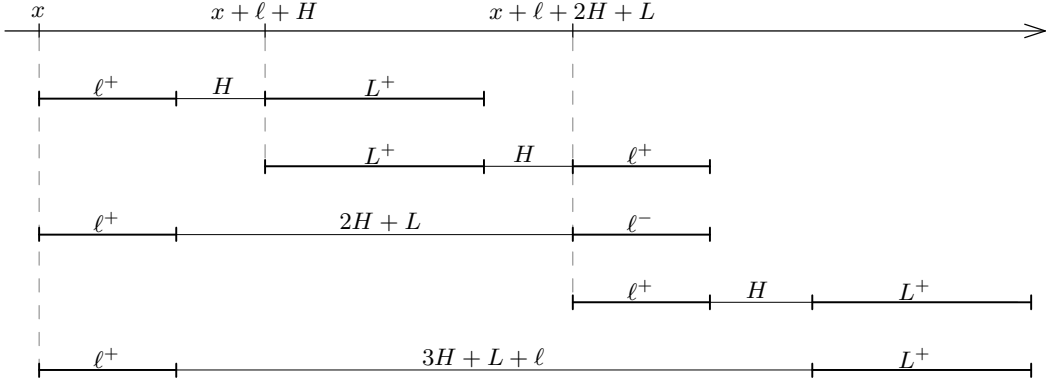

\begin{center}
\begin{texdraw}
\small
\drawdim cm  \setunitscale 0.9
\textref h:C v:C
% asse ascisse
\linewd 0.015
\arrowheadtype t:V
\arrowheadsize l:0.3 w:0.2
\move (-0.5 1) \avec(14.7 1)
% capisaldi su asse ascisse
\move (0 1.1) \lvec(0 0.9) \htext (0 1.3) {$x$}
\move (3.3 1.1) \lvec(3.3 0.9) \htext (3.3 1.3) {$x+\ell+H$}
\move (7.8 1.1) \lvec(7.8 0.9) \htext (7.8 1.3) {$x+\ell+2H+L$}
% verticali 
\linewd 0.01
\setgray 0.5 \lpatt (0.2 0.2)
\move (0 0.9) \lvec(0 -4)
\move (3.3 0.9) \lvec(3.3 -1)
\move (7.8 0.9) \lvec(7.8 -3)
%% step 1
\lpatt()
\setgray 0
\move (0 0) \lvec (6.5 0)
\linewd 0.03
\move (0 0) \lvec (2 0)
\move (3.3 0) \lvec (6.5 0)
\move (0 -.1) \lvec (0 0.1)
\move (2 -.1) \lvec (2 0.1)
\move (3.3 -.1) \lvec (3.3 0.1)
\move (6.5 -.1) \lvec (6.5 0.1)
%\htext (0 -.3) {$x$}
\htext (1 .2) {$\ell^+$}
\htext (2.7 .2) {$H$}
\htext (5 .2) {$L^+$}
%% step 2
\linewd 0.01
\move (3.3 -1) \lvec (9.8 -1)
\linewd 0.03
\move (3.3 -1) \lvec (6.5 -1)
\move (7.8 -1) \lvec (9.8 -1)
\move (3.3 -1.1) \lvec (3.3 -0.9)
\move (6.5 -1.1) \lvec (6.5 -0.9)
\move (7.8 -1.1) \lvec (7.8 -0.9)
\move (9.8 -1.1) \lvec (9.8 -0.9)
\textref h:C v:C
%\htext (3.3 -1.3) {$x+l+H$}
\htext (5 -0.8) {$L^+$}
\htext (7.1 -0.8) {$H$}
\htext (8.8 -0.8) {$\ell^+$}
%% step 3
\linewd 0.01
\move (0 -2) \lvec (9.8 -2)
\linewd 0.03
\move (0 -2) \lvec (2 -2)
\move (7.8 -2) \lvec (9.8 -2)
\move (0 -2.1) \lvec (0 -1.9)
\move (2 -2.1) \lvec (2 -1.9)
\move (7.8 -2.1) \lvec (7.8 -1.9)
\move (9.8 -2.1) \lvec (9.8 -1.9)
\textref h:C v:C
%\htext (0 -2.3) {$x$}
\htext (1 -1.8) {$\ell^+$}
\htext (5 -1.8) {$2H+L$}
\htext (8.8 -1.8) {$\ell^-$}
%% step 4
\linewd 0.01
\move (7.8 -3) \lvec (14.3 -3)
\linewd 0.03
\move (7.8 -3) \lvec (9.8 -3)
\move (11.3 -3) \lvec (14.5 -3)
\move (7.8 -3.1) \lvec (7.8 -2.9)
\move (9.8 -3.1) \lvec (9.8 -2.9)
\move (11.3 -3.1) \lvec (11.3 -2.9)
\move (14.5 -3.1) \lvec (14.5 -2.9)
\textref h:C v:C
%\htext (7.8 -3.3) {$x+l+2H+L$}
\htext (8.8 -2.8) {$\ell^+$}
\htext (10.5 -2.8) {$H$}
\htext (12.8 -2.8) {$L^+$}
%% step 5
\linewd 0.01
\move (0 -4) \lvec (14.5 -4)
\linewd 0.03
\move (0 -4) \lvec (2 -4)
\move (11.3 -4) \lvec (14.5 -4)
\move (0 -4.1) \lvec (0 -3.9)
\move (2 -4.1) \lvec (2 -3.9)
\move (11.3 -4.1) \lvec (11.3 -3.9)
\move (14.5 -4.1) \lvec (14.5 -3.9)
\textref h:C v:C
%\htext (0 -4.3) {$x$}
\htext (1 -3.8) {$\ell^+$}
\htext (6.5 -3.8) {$3H+L+\ell$}
\htext (12.8 -3.8) {$L^+$}
\normalsize
\end{texdraw}
\end{center}
\caption{procedure to reduce to the case where $H>L=\max(L,\ell)$.\label{fig:STEP1}}
\end{figure}

As for the equivalence displayed in \eqref{eq:equiv}, letting $\alpha = \frac{L-\ell}{L+H'}$ 
and $\beta = \frac{L-\ell}{L+H}$ we see that
$\alpha=\frac{\beta}{3-\beta}$ and $\beta= \frac{3\alpha}{1+\alpha}$.
\end{proof}
From now on, $f$ will be a (arbitrarily fixed) continuous function for which one of the two equivalent conditions 
given by \eqref{eq:hpconj} holds. 
For the sake of simplicity, in the next results we will also use the following labelling 
\begin{itemize}
\item[(H1)] $\ds \frac{\ell}{L} \notin \QQ$;
\item[(H2)] $\ds \frac{L-\ell}{2(L+H)} = \frac{n}{m}$, with $n$ and $m$ coprime natural numbers;
\item[($\lnot$H2)] $\ds \frac{L-\ell}{L+H} \notin \QQ$.
\end{itemize}
\begin{lemma}\label{le:1}
If (H1), then $f$ is $(L-\ell)$-periodic.
\end{lemma}
\begin{proof}
Let us define the auxiliary function
\[
\varphi(x) := \int_{x}^{x+(L-\ell)} f(s) ds,\quad \text{for every } x \in \RR;
\]
we are then reduced to prove that $\varphi$ is a constant function.
By assumption, for every $x \in \RR$, there holds
\[
[L^+, H, \ell^+;x] =C \quad \text{and} \quad [\ell^+, H, L^+;x] =C,
\]
hence, subtracting the second equation from the first one (the reader can visualize
such a procedure in Figure \ref{fig:STEP2}), we obtain
\[
[(L-\ell)^+,H-(L-\ell),(L-\ell)^-;x+\ell]=0, \quad \text{for every } x \in \RR, 
\]
or, equivalently, 
\[
\varphi(x) = \varphi(x+H), \quad \text{for every } x \in \RR.
\]
\begin{figure}[t!]
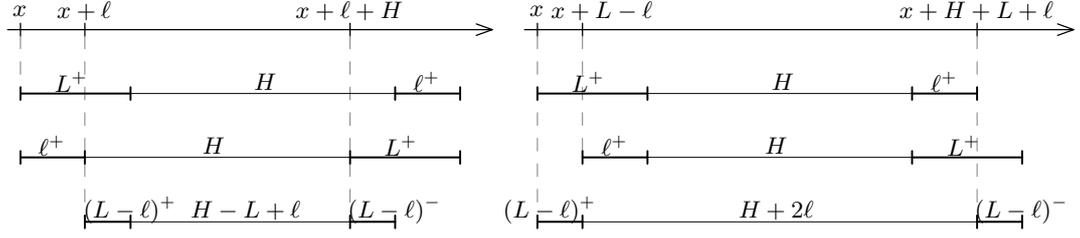

\begin{center}
\begin{texdraw}
\small
\drawdim cm  \setunitscale 0.85
\textref h:C v:C
%%% sinistra %%%
% asse ascisse
\linewd 0.015
\arrowheadtype t:V
\arrowheadsize l:0.25 w:0.18
\move (-0.2 1) \avec(7.3 1)
% capisaldi su asse ascisse
\move (0 1.1) \lvec(0 0.9) \htext (0 1.3) {$x$}
\move (1 1.1) \lvec(1 0.9) \htext (1 1.3) {$x+\ell$}
\move (5.1 1.1) \lvec(5.1 0.9) \htext (5.1 1.3) {$x+\ell+H$}
% verticali 
\linewd 0.01
\setgray 0.5 \lpatt (0.2 0.2)
\move (0 0.9) \lvec(0 0)
\move (1 0.9) \lvec(1 -2)
\move (5.1 0.9) \lvec(5.1 -2)
\lpatt()
\setgray 0
%%% step 1
\move (0 0) \lvec (6.8 0)
\linewd 0.03
\move (0 0) \lvec (1.7 0)
\move (5.8 0) \lvec (6.8 0)
\move (0 -.1) \lvec (0 0.1)
\move (1.7 -.1) \lvec (1.7 0.1)
\move (5.8 -.1) \lvec (5.8 0.1)
\move (6.8 -.1) \lvec (6.8 0.1)
\htext (0.8 .2) {$L^+$}
\htext (3.8 .2) {$H$}
\htext (6.3 .2) {$\ell^+$}
%%% step 2
\linewd 0.01
\move (0 -1) \lvec (6.8 -1)
\linewd 0.03
\move (0 -1) \lvec (1 -1)
\move (5.1 -1) \lvec (6.8 -1)
\move (0 -1.1) \lvec (0 -0.9)
\move (1 -1.1) \lvec (1 -0.9)
\move (5.1 -1.1) \lvec (5.1 -0.9)
\move (6.8 -1.1) \lvec (6.8 -0.9)
\htext (0.5 -0.8) {$\ell^+$}
\htext (3 -0.8) {$H$}
\htext (5.9 -0.8) {$L^+$}
%%% step 3
\linewd 0.01
\move (1 -2) \lvec (5.8 -2)
\linewd 0.03
\move (1 -2) \lvec (1.7 -2)
\move (5.1 -2) \lvec (5.8 -2)
\move (1 -2.1) \lvec (1 -1.9)
\move (1.7 -2.1) \lvec (1.7 -1.9)
\move (5.1 -2.1) \lvec (5.1 -1.9)
\move (5.8 -2.1) \lvec (5.8 -1.9)
\htext (1.7 -1.8) {$(L-\ell)^+$}
\htext (3.5 -1.8) {$H-L+\ell$}
\htext (5.8 -1.8) {$(L-\ell)^-$}
%%%%%%%%%%%%%%%%%%%%%%%%%%%%%%
%%% destra %%%
% asse ascisse
\linewd 0.015
\arrowheadtype t:V
\arrowheadsize l:0.25 w:0.18
\move (7.8 1) \avec(16.3 1)
% capisaldi su asse ascisse
\move (8 1.1) \lvec(8 0.9) \htext (8 1.3) {$x$}
\move (8.7 1.1) \lvec(8.7 0.9) \htext (9 1.3) {$x+L-\ell$}
\move (14.8 1.1) \lvec(14.8 0.9) \htext (14.8 1.3) {$x+H+L+\ell$}
% verticali 
\linewd 0.01
\setgray 0.5 \lpatt (0.2 0.2)
\move (8 0.9) \lvec(8 -2)
\move (8.7 0.9) \lvec(8.7 -1)
\move (14.8 0.9) \lvec(14.8 -2)
\lpatt()
\setgray 0
%%% step 1
\move (8 0) \lvec (14.8 0)
\linewd 0.03
\move (8 0) \lvec (9.7 0)
\move (13.8 0) \lvec (14.8 0)
\move (8 -.1) \lvec (8 0.1)
\move (9.7 -.1) \lvec (9.7 0.1)
\move (13.8 -.1) \lvec (13.8 0.1)
\move (14.8 -.1) \lvec (14.8 0.1)
\htext (8.8 .2) {$L^+$}
\htext (11.8 .2) {$H$}
\htext (14.3 .2) {$\ell^+$}
%% step 2
\linewd 0.01
\move (8.7 -1) \lvec (15.5 -1)
\linewd 0.03
\move (8.7 -1) \lvec (9.7 -1)
\move (13.8 -1) \lvec (15.5 -1)
\move (8.7 -1.1) \lvec (8.7 -0.9)
\move (9.7 -1.1) \lvec (9.7 -0.9)
\move (13.8 -1.1) \lvec (13.8 -0.9)
\move (15.5 -1.1) \lvec (15.5 -0.9)
\htext (9.2 -0.8) {$\ell^+$}
\htext (11.7 -0.8) {$H$}
\htext (14.6 -0.8) {$L^+$}
%%%% step 3
\linewd 0.01
\move (8 -2) \lvec (15.5 -2)
\linewd 0.03
\move (8 -2) \lvec (8.7 -2)
\move (14.8 -2) \lvec (15.5 -2)
\move (8 -2.1) \lvec (8 -1.9)
\move (8.7 -2.1) \lvec (8.7 -1.9)
\move (14.8 -2.1) \lvec (14.8 -1.9)
\move (15.5 -2.1) \lvec (15.5 -1.9)
\htext (8.2 -1.8) {$(L-\ell)^+$}
\htext (11.7 -1.8) {$H+2\ell$}
\htext (15.5 -1.8) {$(L-\ell)^-$}
\normalsize
\end{texdraw}
\end{center}
\caption{the picture at left represents the procedure to obtain the $H$-periodicity for the 
function $\varphi$; at right, the one for the $(H+L+\ell)$-periodicity.\label{fig:STEP2}}
\end{figure}
On the other hand, the assumption can be read as
\[
[L^+, H, \ell^+;x] =C \quad \text{and} \quad [\ell^+, H, L^+;x+(L-\ell)] =C,
\]
implying
\(
[(L-\ell)^+,H+2\ell,(L-\ell)^-;x]=0,
\)
for every $x \in \RR$, or, equivalently,
\[
\varphi(x) = \varphi(x+H+L+\ell), \quad \text{for every } x \in \RR.
\]
Since $(H+L+\ell)/H$ is not rational, and the continuous function $\varphi$ 
is at the same time $H$-periodic and $(H+L+\ell)$-periodic, then $\varphi$ is necessarily constant.
\end{proof}
\begin{lemma}\label{le:2}
If ($\lnot$H2), then $f$ is $(L+\ell)$-periodic.
\end{lemma}
\begin{proof} 
Thanks to Lemma \ref{lem:hole}, we may assume (up to replacing $H$ with $H'=3H+L+\ell$) that $H>L$;
notice, in particular, that $H'$ still satisfies ($\lnot$H2) once $H$ does.

Let us define the auxiliary function
\[
\psi(x) := \int_{x}^{x+(L+\ell)} f(s) ds,\quad \text{for every } x \in \RR:
\]
we claim that $\psi$ is at the same time 
$2(L+H)-$ and $2(\ell+H)$-periodic. Since assumption ($\lnot$H2) is equivalent to $\frac{\ell+H}{L+H} \notin \QQ$
(just write $\frac{\ell+H}{L+H} = 1-\frac{L-\ell}{L+H}$), we will then conclude that $\psi$ is a constant function,
hence $f$ is $(L+\ell)$-periodic. 

\begin{figure}[t!]
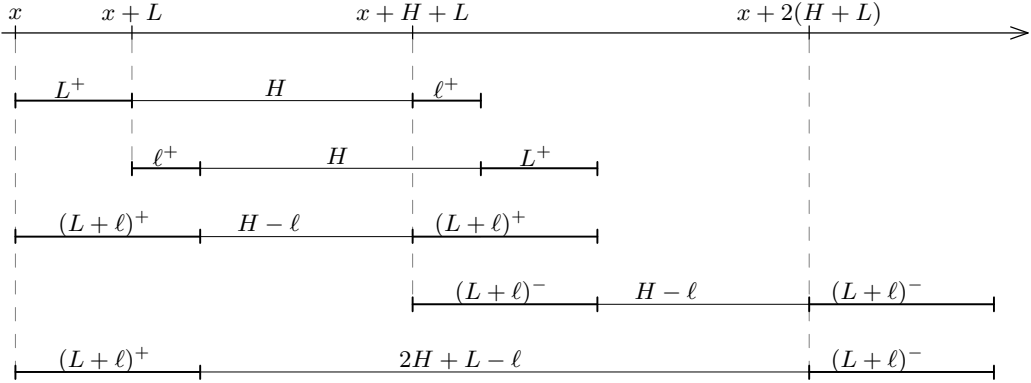

\begin{center}
\begin{texdraw}
\small
\drawdim cm  \setunitscale 0.9
\textref h:C v:C
%%% sinistra %%%
% asse ascisse
\linewd 0.015
\arrowheadtype t:V
\arrowheadsize l:0.25 w:0.18
\move (-0.2 1) \avec(14.8 1)
% capisaldi su asse ascisse
\move (0 1.1) \lvec(0 0.9) \htext (0 1.3) {$x$}
\move (1.7 1.1) \lvec(1.7 0.9) \htext (1.7 1.3) {$x+L$}
\move (5.8 1.1) \lvec(5.8 0.9) \htext (5.8 1.3) {$x+H+L$}
\move (11.6 1.1) \lvec(11.6 0.9) \htext (11.6 1.3) {$x+2(H+L)$}
% verticali 
\linewd 0.01
\setgray 0.5 \lpatt (0.2 0.2)
\move (0 0.9) \lvec(0 -4)
\move (1.7 0.9) \lvec(1.7 -1)
\move (5.8 0.9) \lvec(5.8 -3)
\move (11.6 0.9) \lvec(11.6 -4)
\lpatt()
\setgray 0
%%% step 1
\move (0 0) \lvec (6.8 0)
\linewd 0.03
\move (0 0) \lvec (1.7 0)
\move (5.8 0) \lvec (6.8 0)
\move (0 -.1) \lvec (0 0.1)
\move (1.7 -.1) \lvec (1.7 0.1)
\move (5.8 -.1) \lvec (5.8 0.1)
\move (6.8 -.1) \lvec (6.8 0.1)
\htext (0.8 .2) {$L^+$}
\htext (3.8 .2) {$H$}
\htext (6.3 .2) {$\ell^+$}
%%% step 2
\linewd 0.01
\move (1.7 -1) \lvec (8.5 -1)
\linewd 0.03
\move (1.7 -1) \lvec (2.7 -1)
\move (6.8 -1) \lvec (8.5 -1)
\move (1.7 -1.1) \lvec (1.7 -0.9)
\move (2.7 -1.1) \lvec (2.7 -0.9)
\move (6.8 -1.1) \lvec (6.8 -0.9)
\move (8.5 -1.1) \lvec (8.5 -0.9)
\htext (2.2 -0.8) {$\ell^+$}
\htext (4.7 -0.8) {$H$}
\htext (7.6 -0.8) {$L^+$}
%%% step 3
\linewd 0.01
\move (0 -2) \lvec (6.8 -2)
\linewd 0.03
\move (0 -2) \lvec (2.7 -2)
\move (5.8 -2) \lvec (8.5 -2)
\move (0 -2.1) \lvec (0 -1.9)
\move (2.7 -2.1) \lvec (2.7 -1.9)
\move (5.8 -2.1) \lvec (5.8 -1.9)
\move (8.5 -2.1) \lvec (8.5 -1.9)
\htext (1.3 -1.8) {$(L+\ell)^+$}
\htext (3.7 -1.8) {$H-\ell$}
\htext (6.8 -1.8) {$(L+\ell)^+$}
%%% step 4
\linewd 0.01
\move (5.8 -3) \lvec (14.3 -3)
\linewd 0.03
\move (5.8 -3) \lvec (8.5 -3)
\move (11.6 -3) \lvec (14.3 -3)
\move (5.8 -3.1) \lvec (5.8 -2.9)
\move (8.5 -3.1) \lvec (8.5 -2.9)
\move (11.6 -3.1) \lvec (11.6 -2.9)
\move (14.3 -3.1) \lvec (14.3 -2.9)
\htext (7.1 -2.8) {$(L+\ell)^-$}
\htext (9.5 -2.8) {$H-\ell$}
\htext (12.6 -2.8) {$(L+\ell)^-$}
%%% step 5
\linewd 0.01
\move (0 -4) \lvec (14.3 -4)
\linewd 0.03
\move (0 -4) \lvec (2.7 -4)
\move (11.6 -4) \lvec (14.3 -4)
\move (0 -4.1) \lvec (0 -3.9)
\move (2.7 -4.1) \lvec (2.7 -3.9)
\move (11.6 -4.1) \lvec (11.6 -3.9)
\move (14.3 -4.1) \lvec (14.3 -3.9)
\htext (1.3 -3.8) {$(L+\ell)^+$}
\htext (6.5 -3.8) {$2H+L-\ell$}
\htext (12.6 -3.8) {$(L+\ell)^-$}
\normalsize
\end{texdraw}
\end{center}
\caption{procedure to obtain the $(L+H)$-periodicity of $\psi$. \label{fig:figura_le2}}
\end{figure}

By assumption, we see that
\[
[L^+, H, \ell^+;x] = C \quad \text{and} \quad [\ell^+, H, L^+;x+L)] =C, \quad \text{for every } x \in \RR;
\]
taking into account that $H>L$, we may sum both terms of the above equalities (see Figure \ref{fig:figura_le2})
to obtain
\begin{equation} \label{eq:adhoc1}
[(L+\ell)^+, H-\ell , (L+\ell)^+;x] = 2C, \quad \text{for every } x \in \RR.
\end{equation}
By translating of $L+H$ the above equality and changing the sign, 
we see that
\begin{equation} \label{eq:adhoc2}
[(L+\ell)^-, H-\ell , (L+\ell)^-;x+L+H] = -2C, \quad \text{for every } x \in \RR;
\end{equation}
again, summing both terms of \eqref{eq:adhoc1} and \eqref{eq:adhoc2}, it follows that
\[
[(L+\ell)^+, 2H+L-\ell , (L+\ell)^-;x] = 0, \quad \text{for every } x \in \RR,
\]
which is equivalent to the $2(H+L)$ periodicity of $\psi$.
Swapping $L$ with $\ell$ we obtain the $2(H+\ell)$ periodicity of $\psi$.
\end{proof}
\begin{lemma}\label{le:4}
If (H1), then $f$ is $(L+H)$-periodic.
\end{lemma}
\begin{proof}
Since, by Lemma \ref{le:1}, $f$ is $(L-\ell)$-periodic, there exists a constant $k \in \RR$ such that 
$\left(\int_x^{x+L-\ell}f(s) ds\right)/(L-\ell) = k$, for any $x \in \RR$. 
Let us define the $(L-\ell)$-periodic function
\[
g(x) = f(x) - k, \qquad x \in \RR;
\]
By definition, $\int_x^{x+L-\ell}g = 0$ for any $x \in \RR$; furthermore
$g$ satisfies the assumption of conjecture \eqref{eq:PC}, indeed 
(since $f$ satisfies such assumption for some constant $C \in \RR$)
\[
\int_{\sigma(I)}g(s) ds = \int_{\sigma(I)}f(s) ds - k (L+\ell) = C - k (L+\ell).
\]
We term $C_g =C - k (L+\ell)$ and we plan to prove that $g$ (and, as a consequence, $f$) is $(L+H)$-periodic.

Let us consider the set 
\[
\Delta = \left\{ mL+n(L-\ell) : m,n \in \ZZ \right\};
\]
since, by assumption, $\frac{L}{L-\ell}$ is not rational,
the set $\Delta$ is dense in $\RR$. It turns out 
that for every $m,n \in \ZZ$, denoting $t = mL+n(L-\ell)$ the corrisponding element in $\Delta$,
there holds
\begin{equation}\label{eq:dadimostrare}
\int_{x}^{x+t} g(s) ds + \int_{x+L+H}^{x+L+H+t} g(s) ds = mC_g,\quad \text{for every } x \in \RR.
\end{equation}
Differentiating by $x$ the previous equation, we obtain that the continuous function 
\[
h_x(t) = g(x+t)-g(x)+g(x+L+H+t)-g(x+L+H), \quad t \in \RR,
\]
vanishes on the dense set $\Delta \subset \RR$, hence it vanishes on $\RR$ and
\[
g(x+L+H+t)-g(x+L+H) = -\left[ g(x+t)-g(x) \right], \quad \text{for every } t,x \in \RR.
\]
From the last equation, we deduce, for every $t,x \in \RR$, the following integral relation, 
\[
\begin{split}
\int_{x+L+H}^{x+L+H+t}[g(s)-g(x+L+H)]ds 
    &= \int_{x}^{x+t}[g(s'+L+H)-g(x+L+H)]ds' \\
    =& \int_{x}^{x+t}[g((s'-x)+x+L+H)-g(x+L+H)]ds' \\
    =& -\int_{x}^{x+t}[g(s')-g(x)]ds'
\end{split}
\]
that can be written as
\[
\int_{x+L+H}^{x+L+H+t}g(s)ds +\int_{x}^{x+t}g(s)ds = [g(x+L+H) +g(x)]t, \quad \text{for every } t,x \in \RR.
\]
We now deduce the $(L+H)$-periodicity of $g$ comparing the previous relation 
with equation \eqref{eq:dadimostrare}. 
Indeed, for every $x \in \RR$ and $t =  mL+n(L-\ell) \in \Delta$
there holds
\[
 [g(x+L+H) +g(x)]t = mC_g,
\]
we conclude choosing  $m=0$ and $n\neq 0$.\\
We are left to prove equation \eqref{eq:dadimostrare}. 
Equivalently, we are going to prove such relation by induction on $m \in \ZZ$.
When $m=0$ the equation follows from the fact that $\int_a^{a+L-\ell}g=0$, for every $a \in \RR$.
When $m>0$, we have the equalities
\[
\begin{split}
\int_{x+L+H}^{x+L+H+(m+1)L+n(L-\ell)} g(s) ds 
& = 
\int_{x+L+H}^{x+L+H+mL+n(L-\ell)} g(s) ds \\
& \qquad \qquad + \int_{x+L+H+mL+n(L-\ell)}^{x+L+H+(m+1)L+n(L-\ell)} g(s) ds \\
& = - \int_{x}^{x+mL+n(L-\ell)} g(s) ds \\
& \qquad \qquad + mC_g + \int_{x+L+H+mL+n(L-\ell)}^{x+L+H+(m+1)L+n(L-\ell)} g(s) ds .
\end{split}
\]
Thus we will be done if we can prove that
\[
\begin{split}
\int_{x+L+H+mL+n(L-\ell)}^{x+L+H+(m+1)L+n(L-\ell)} g(s) ds 
& = C_g + \int_{x}^{x+mL+n(L-\ell)} g(s) ds - \int_{x}^{x+(m+1)L+n(L-\ell)} g(s) ds \\
& = C_g - \int_{x+mL+n(L-\ell)}^{x+(m+1)L+n(L-\ell)} g(s) ds,
\end{split}
\]
that is equivalent to (replacing $a =x+mL+n(L-\ell)$)
\[
\int_{a+L+H}^{a+L+H+L} g(s) ds + \int_{a}^{a+L} g(s) ds = C_g.
\]
Since $g$ has vanishing integral on interval with lenght $L-\ell$, the first integral in
the previous equality is $\int_{a+L+H}^{a+L+H+\ell} g(s) ds$; 
the equality holds true by definition of the constant $C_g$.
To conclude our proof we still have to consider the case where $m<0$. 
Replacing $x'=x+mL+n(L-\ell)$ the left hand side of equation \eqref{eq:dadimostrare}
reads as
\[
- \int_{x'+L+H}^{x'+L+H-mL-n(L-\ell)}g(s) ds - \int_{x'}^{x'-mL-n(L-\ell)}g(s) ds
= -\left( -mC_g \right),
\]
using what we have already proved.
\end{proof}
\begin{lemma}\label{le:6}
If  (H1) and (H2) with $m$ even, then $f$ is $\ell$-periodic.
\end{lemma}
\begin{proof}
By Lemmata \ref {le:1} and \ref{le:4} $f$ is both $(L-\ell)$- and $(L+H)$-periodic. 
Let $g$ be the translation of $f$ introduced in Lemma \ref{le:4}.
We plan to prove that $g$ (and, as a consequence, $f$) is $\ell$-periodic.
Indeed, since $g$ is both $(L+H)$-periodic and $(L-\ell)$-periodic, we have for every $x \in \RR$
\begin{multline*}
C_g = \int_{x}^{x+L} g(s) ds + \int_{x+L+H}^{x+L+H+\ell} g(s) ds \\= 
      \int_{x}^{x+\ell} g(s) ds + \int_{x+\ell}^{x+L} g(s) ds + \int_{x}^{x+\ell} g(s) ds = 
      2 \int_{x}^{x+\ell} g(s) ds.
\end{multline*}
Hence $\int_{x}^{x+\ell} g(s) ds$ is constant and, by differentiation, we 
obtain that $g$ (hence $f$) is $\ell$-periodic.
\end{proof}
\begin{theorem}\label{teo:nec_suf}
Let $I=[a,b]\cup[c,d]$ for some $a<b<c<d$, and $\Sigma' = \Sigma$.
Then conjecture \eqref{eq:PC} holds for $f \in \cont(\RR)$ 
if and only if 
\begin{equation} \label{eq:condNS}
\text{(H1) $\wedge$  \Big( ($\lnot$H2) $\vee$ \big( (H2) with m even \big) \Big)}.
\end{equation}
\end{theorem}
\begin{proof}[Proof of the necessary condition]
Assume \eqref{eq:condNS} fails. As is easy to check, this means that 
\begin{equation} \label{eq:condNSneg}
\text{$\lnot$(H1) $\vee$  \big( (H2) with m odd \big)}.
\end{equation}
In order to disprove \eqref{eq:PC}, let $C \in \RR$ be arbitrary.
Now, if $\lnot$(H1), then there exist $s \in \RR$ and $n_1, n_2 \in \NN$ such that 
\[
L = n_1 s\quad \text{and} \quad \ell = n_2 s;
\]
letting $f$ to be any $s$-periodic, non-constant continuous function such that 
\[
\int_{0}^s f(x) dx = \frac{C}{n_1+n_2},
\]
we see that
\[
\int_{\sigma(I)} f(x) dx = (n_1+n_2)\int_{0}^s f(x) dx  = C.
\]

If, on the contrary, $\frac{L-\ell}{2(L+H)} = \frac{n}{2h+1}$ for some $n,h \in \NN$,
then $L-\ell = ns$ and $2(L+H) = (2h+1)s$ for some $s>0$. 
In this case, 
\[
f(x) = \sin \left( \frac{2\pi}{s}x \right)+\frac{C}{L+\ell} 
\]
turns out to be the required non-constant function which contradicts \eqref{eq:PC}.
\end{proof}
\begin{proof}[Proof of the sufficient condition]
On the one hand, if (H1) and ($\lnot$H2) hold then, by Lemmata \ref{le:1} and \ref{le:2}, 
$f$ is both $(L-\ell)$- and $(L+\ell)$-periodic. Since, by assumption (H1), 
$\frac{L-\ell}{L+\ell} \notin \QQ$, $f$ is necessarily a constant function.
On the other hand, if (H1) and (H2) with $m$ even hold then $f$ turns out to be 
$(L-\ell)$-periodic, by Lemma \ref{le:1}, and $\ell$-periodic by Lemma \ref{le:6}.
Also in this case $f$ is necessarily constant, indeed  $\frac{L-\ell}{\ell} \notin \QQ$.
\end{proof}

\begin{remark}
Let us observe that conjecture \eqref{eq:PC} holds whenever the length of the  hole between the two intervals 
coincide with one of their lengths, whose ratio is irrational (i.e. $\ell/L \notin \QQ$ and $H \in \{\ell,L \}$).
\end{remark}

%===============================================
\section{One more interval, much more complexity}
\label{sec:2}
%============================================

\begin{proposition}
Let $\Sigma'=\Sigma$, $f \in \cont(\RR)$ and let $\ell, L, H, \mathcal{L} > 0$ be such that
\[
\frac{\ell + L + \mathcal{L}}{H} \in \QQ.
\]
Then conjecture \eqref{eq:PC} does not hold for
the three-interval set 
\[
I = [0,\ell] \cup [\ell+H,\ell+H+L] \cup [\ell+L+2H,\ell+L+\mathcal{L}+2H].
\]
\end{proposition}
\begin{proof}
Let $C>0$ be arbitrarily fixed. 
By assumption, there exist $s \in \RR$ and $n_1,n_2 \in \NN$ such that
\[
\ell + L + \mathcal{L} = n_1s \quad \text{and} \quad H = n_2s;
\]
we now choose an arbitrary (non-constant) $s$-periodic function $f$ such that 
\[
\int_0^s f(x)dx = \frac{C}{n_1}.
\]
For any $\sigma \in \Sigma$ the following identity holds
\[
\int_{\sigma(I)} f(x)dx = \int_{x_{\sigma}}^{x_{\sigma}+\ell+L+\mathcal{L}+2H} f(x) dx -
                          \int_{y_{\sigma}}^{y_{\sigma}+H} f(x) dx -
                          \int_{z_{\sigma}}^{z_{\sigma}+H} f(x) dx
\]
for suitable $x_\sigma, y_{\sigma}, z_{\sigma} \in \RR$, hence 
\[
\int_{\sigma(I)} f(x)dx = (n_1+2n_2)\frac{C}{n_1} - 2n_2\frac{C}{n_1} = C.
\]
\end{proof}
\begin{remark}
Let us observe that, given $\ell, L, H > 0$, there exist infinitely many $\mathcal{L}>0$
satisfying the assumption of the previous proposition. This fact implies, in particular, that
given \emph{any} two-interval set (for which the Pompeiu conjecture may fails or not)
we can add a third interval to obtain a set for which the conjecture \eqref{eq:PC} does not hold.
\end{remark}
We conclude with an example of a
three-interval set for which the Pompeiu conjecture holds; 
notice that the first two intervals of such a set constitute a domain which 
does not enjoy the Pompeiu property. 
\begin{example}\label{example}
Let us consider $\ell, L >0$ be such that $\frac{\ell}{L} \notin \QQ$, and the three-interval set
\[
I = [0,\ell] \cup [2\ell,3\ell] \cup [4\ell,4\ell+L].
\]
Let $f$ be a continuous function such that its integral on every $\sigma(I)$, $\sigma \in \Sigma$, 
is constantly equal to $C$, for some $C>0$.
For the sake of simplicity,
we extend in a natural way the notation for the two-interval set, introduced in the previous section,
to the three-interval case. 
Since $\Sigma' = \Sigma$ we obtain, for every $x \in \RR$,
\[
[\ell^+,\ell,\ell^+,\ell,L^+;x+L] = C 
\quad \text{and} \quad
[L^+,\ell,\ell^+,\ell,\ell^+;x] = C 
\] 
hence summing term by term (see Figure \ref{fig:ultima}, first and second lines)
\[
\int_x^{x+4\ell+2L}f(s)ds = 2C, \quad \forall x \in \RR,
\]
which implies that $f$ is $(4\ell+2L)$-periodic.\par
On the other hand since, for every $x \in \RR$,
\[
[\ell^-,\ell,\ell^-,\ell,L^-;x+\ell+L] = -C 
\quad \text{and} \quad
[L^+,\ell,\ell^+,\ell,\ell^+;x] = C 
\] 
we obtain $[L^+,5\ell,L^-;x] = 0$, for every $x \in \RR$ 
(see Figure \ref{fig:ultima}, third and second lines), which implies that the function
\[
\varphi(x) = \int_x^{x+L}f(s) ds, \qquad x \in \RR
\]
is not only $(4\ell+2L)$-periodic (as $f$ is), but also $(5\ell+L)$-periodic. 
Since $\frac{\ell}{L} \notin \QQ$, not even $\frac{4\ell+2L}{5\ell+L}$ does, $\varphi$ is constant and 
$f$ is $\ell$-periodic. 
Since $\frac{4\ell+2L}{\ell} \notin \QQ$, $f$ is constant.
\end{example}

\begin{figure}[ht]
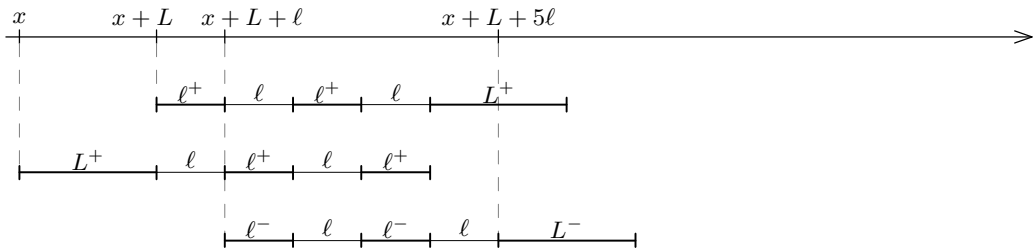

\begin{center}
\begin{texdraw}
\small
\drawdim cm  \setunitscale 0.9
\textref h:C v:C
%%% sinistra %%%
% asse ascisse
\linewd 0.015
\arrowheadtype t:V
\arrowheadsize l:0.25 w:0.18
\move (-0.2 1) \avec(14.8 1)
% capisaldi su asse ascisse
\move (0 1.1) \lvec(0 0.9) \htext (0 1.3) {$x$}
\move (2 1.1) \lvec(2 0.9) \htext (1.8 1.3) {$x+L$}
\move (3 1.1) \lvec(3 0.9) \htext (3.4 1.3) {$x+L+\ell$}
\move (7 1.1) \lvec(7 0.9) \htext (7 1.3) {$x+L+5\ell$}
% verticali 
\linewd 0.01
\setgray 0.5 \lpatt (0.2 0.2)
\move (0 0.9) \lvec(0 -1)
\move (2 0.9) \lvec(2 0)
\move (3 0.9) \lvec(3 -2)
\move (7 0.9) \lvec(7 -2)
\lpatt()
\setgray 0
%%% step 1
\linewd 0.01
\move (2 0) \lvec (8 0)
\linewd 0.03
\move (2 0) \lvec (3 0)
\move (4 0) \lvec (5 0)
\move (6 0) \lvec (8 0)
\move (2 -.1) \lvec (2 0.1)
\move (3 -.1) \lvec (3 0.1)
\move (4 -.1) \lvec (4 0.1)
\move (5 -.1) \lvec (5 0.1)
\move (6 -.1) \lvec (6 0.1)
\move (8 -.1) \lvec (8 0.1)
\htext (2.5 .2) {$\ell^+$}
\htext (3.5 .2) {$\ell$}
\htext (4.5 .2) {$\ell^+$}
\htext (5.5 .2) {$\ell$}
\htext (7 .2) {$L^+$}
%%% step 2
\linewd 0.01
\move (0 -1) \lvec (6 -1)
\linewd 0.03
\move (0 -1) \lvec (2 -1)
\move (3 -1) \lvec (4 -1)
\move (5 -1) \lvec (6 -1)
\move (0 -1.1) \lvec (0 -0.9)
\move (2 -1.1) \lvec (2 -0.9)
\move (3 -1.1) \lvec (3 -0.9)
\move (4 -1.1) \lvec (4 -0.9)
\move (5 -1.1) \lvec (5 -0.9)
\move (6 -1.1) \lvec (6 -0.9)
\htext (1 -0.8) {$L^+$}
\htext (2.5 -0.8) {$\ell$}
\htext (3.5 -0.8) {$\ell^+$}
\htext (4.5 -0.8) {$\ell$}
\htext (5.5 -0.8) {$\ell^+$}
%%%% step 3
\linewd 0.01
\move (3 -2) \lvec (9 -2)
\linewd 0.03
\move (3 -2) \lvec (4 -2)
\move (5 -2) \lvec (6 -2)
\move (7 -2) \lvec (9 -2)
\move (3 -2.1) \lvec (3 -1.9)
\move (4 -2.1) \lvec (4 -1.9)
\move (5 -2.1) \lvec (5 -1.9)
\move (6 -2.1) \lvec (6 -1.9)
\move (7 -2.1) \lvec (7 -1.9)
\move (9 -2.1) \lvec (9 -1.9)
\htext (3.5 -1.8) {$\ell^-$}
\htext (4.5 -1.8) {$\ell$}
\htext (5.5 -1.8) {$\ell^-$}
\htext (6.5 -1.8) {$\ell$}
\htext (8 -1.8) {$L^-$}
\normalsize
\end{texdraw}
\end{center}
\caption{procedure to prove that function $f$ in Example \ref{example} is constant.\label{fig:ultima}}
\end{figure}

%======================================
%\bibliography{vivinabibliog}

\begin{thebibliography}{1}

\bibitem{Brown1982}
L.~Brown and J.-P. Kahane.
\newblock A note on the {P}ompeiu problem for convex domains.
\newblock {\em Math. Ann.}, 259(1):107--110, 1982.

\bibitem{Brown1968}
L.~Brown, F.~Schnitzer, and A.~L. Shields.
\newblock A note on a problem of {D}. {P}ompeiu.
\newblock {\em Math. Z.}, 105:59--61, 1968.

\bibitem{Brown1973}
L.~Brown, B.~M. Schreiber, and B.~A. Taylor.
\newblock Spectral synthesis and the {P}ompeiu problem.
\newblock {\em Ann. Inst. Fourier (Grenoble)}, 23(3):125--154, 1973.

\bibitem{Pompeiu1}
D.~Pompeiu.
\newblock Sur certain systèmes d'équationslinéaires et sur une propriété
  intégrale de functions de plusieur variables.
\newblock {\em C. R. Acam. Sci. Paris}, 1088:1138--1139, 1929.

\bibitem{Pompeiu3}
D.~Pompeiu.
\newblock Sur une propriété de functions continue dépendent de plusieur
  variables.
\newblock {\em Bull. Sci. Math.}, 53(2):328--332, 1929.

\bibitem{Pompeiu2}
D.~Pompeiu.
\newblock Sur une propriété intégrale de functions de deux variables
  reélles.
\newblock {\em Bull. Sci. Acad. Royale Belgique}, 15(5):265--269, 1929.

\bibitem{Zalcman1992}
L.~Zalcman.
\newblock A bibliographic survey of the {P}ompeiu problem.
\newblock In {\em Approximation by solutions of partial differential equations
  ({H}anstholm, 1991)}, volume 365 of {\em NATO Adv. Sci. Inst. Ser. C Math.
  Phys. Sci.}, pages 185--194. Kluwer Acad. Publ., Dordrecht, 1992.

\end{thebibliography}
%\bibliographystyle{plain}
%======================================

%=====================================
\end{document}